\documentclass{amsart}

\usepackage{enumerate}
\usepackage[dvipsnames]{xcolor}
\usepackage{hyperref}
\hypersetup{
    colorlinks=true,
    linkcolor=blue,
    filecolor=blue,      
    urlcolor=blue,
    citecolor=blue,
}
\usepackage{amsmath,amssymb,amsrefs}
\usepackage{tikz}
\usetikzlibrary{positioning, cd, arrows}
\usepackage{enumerate}
\usepackage{bm}
\usepackage{float}
\usepackage{wrapfig}
\usepackage{caption}
\usepackage{subcaption}
\usepackage{mathrsfs}

\usepackage[many]{tcolorbox}
\tcbuselibrary{listings}

\usetikzlibrary{graphs}
\usetikzlibrary{graphs.standard}

\newlength\inwd
\newcounter{notebookcounter}

\newfloat{lstfloat}{htbp}{lop}
\floatname{lstfloat}{Listing}
\numberwithin{lstfloat}{section}

\colorlet{inputcolor}{blue!50!black}
\colorlet{outputcolor}{red!50!black}

\lstdefinelanguage{genericnotebookin}{
    morekeywords={print, def, return},
    morecomment=[l]{//},
    basicstyle=\footnotesize\ttfamily,
    keywordstyle=\color{green!40!black}\bfseries,
    commentstyle=\color{purple!40!black}\itshape,
}

\lstdefinelanguage{genericnotebookout}{
    morekeywords={:,Type},
    morecomment=[l]{//},
    basicstyle=\footnotesize\ttfamily,
    keywordstyle=\bfseries\color{green!40!black},
    commentstyle=\itshape\color{purple!40!black},
}

\newtcblisting{notebookin}[1][\thenotebookcounter]{
  enlarge left by=\inwd,
  width=\linewidth-\inwd,
  enhanced,
  boxrule=0.4pt,
  colback=ForestGreen!15,
  listing only,
  top=0pt,
  bottom=0pt,
  overlay={
    \node[
      anchor=north east,
      text width=\inwd,
      font=\footnotesize\ttfamily\color{inputcolor},
      inner ysep=2mm,
      inner xsep=0pt,
      outer sep=0pt
      ] 
      at (frame.north west)
      {\stepcounter{notebookcounter}In [#1]:};
  }
  listing engine=listing,
  listing options={
    aboveskip=1pt,
    belowskip=1pt,
    language=genericnotebookin,
    showstringspaces=false,
    breaklines=true,
  },
}
\newtcblisting{notebookout}[1][\thenotebookcounter]{
  enlarge left by=\inwd,
  width=\linewidth-\inwd,
  enhanced,
  boxrule=0.1pt,colback=White,
  listing only,
  top=0pt,
  bottom=0pt,
  overlay={
    \node[
      anchor=north east,
      text width=\inwd,
      font=\footnotesize\ttfamily\color{outputcolor},
      inner ysep=2mm,
      inner xsep=0pt,
      outer sep=0pt
      ] 
      at (frame.north west)
      {Out [#1]:};
  }
  listing engine=listing,
  listing options={
    aboveskip=1pt,
    belowskip=1pt,
    language=genericnotebookout,
    showstringspaces=false,
  },
}

\tikzset{Rightarrow/.style={double equal sign distance,>={Implies},->},
triple/.style={-,preaction={draw,Rightarrow}},
quadruple/.style={preaction={draw,Rightarrow,shorten >=0pt},shorten >=1pt,-,double,double
distance=0.2pt}}

\floatstyle{ruled}
\newfloat{typefloat}{htbp}{lop}
\floatname{typefloat}{Type}

\usepackage[draft]{pdfcomment}
\usepackage{xparse}
    \DeclareDocumentCommand \EDITmargin { o m } {
        \IfNoValueTF {#1} {
            \pdfmargincomment[icon=Note]{#2}
        }{
            \pdfmargincomment[icon=NOte,author=#1]{#2}
        }
    }
    \DeclareDocumentCommand \EDITcomment { o m } {
        \IfNoValueTF {#1} {
            \pdfcomment[color=Blue!20]{#2}
        }{
            \pdfcomment[color=Blue!20,author=#1]{#2}
        }
    }
    \DeclareDocumentCommand \EDITalt { o m m } {
        \IfNoValueTF {#1} {
            \pdfmarkupcomment[markup=StrikeOut]{#2}{#3}
        }{
            \pdfmarkupcomment[markup=StrikeOut,icon=key, author=#1]{#2}{#3}
        }
    }
    \DeclareDocumentCommand \EDITtypo { o m o } {
        \IfNoValueTF {#3} {
            \pdfmarkupcomment[markup=Squiggly,author=#1]{#2}{}
        }{
            \pdfmarkupcomment[markup=Squiggly,author=#1]{#2}{#3}
        }
    }
    \DeclareDocumentCommand \EDIThighlight { o m m } {
\pdfmarkupcomment[markup=Highlight,author=#1,color=blue!20]{#2}{#3}
}

\DeclareMathOperator{\GL}{GL}

\DeclareMathOperator{\id}{id}

\newcommand{\N}{\mathbb{N}}
\newcommand{\Z}{\mathbb{Z}}

\tikzcdset{
  cells={font=\everymath\expandafter{\the\everymath\displaystyle}},
}
\makeatletter
\newcommand{\oset}[3][0ex]{\mathrel{\mathop{#3}\limits^{
    \vbox to#1{\kern-2\ex@
    \hbox{$\scriptstyle#2$}\vss}}}}
\makeatother

\newcommand{\Epi}[1]{\oset{\twoheadrightarrow}{\cat{#1}}}\newcommand{\epi}[1]{\Epi{#1}}
\newcommand{\Mono}[1]{\oset{\hookrightarrow}{\cat{#1}}}
\newcommand{\mono}[1]{\Mono{#1}}

\newcommand{\Cat}[1]{\mathtt{#1}}
\newcommand{\cat}[1]{\Cat{#1}}

\renewcommand{\leq}{\leqslant}
\renewcommand{\geq}{\geqslant}
\renewcommand{\preceq}{\preccurlyeq}

\newcommand{\inputref}[2]{\hyperref[#1]{\texttt{{\color{inputcolor}In\,[#2]}}}}
\newcommand{\outputref}[2]{\hyperref[#1]{\texttt{{\color{outputcolor}Out\,[#2]}}}}

\makeatletter
\newcommand{\xRrightarrow}[2][]{\ext@arrow 0359\Rrightarrowfill@{#1}{#2}}
\newcommand{\Rrightarrowfill@}{\arrowfill@\equiv\equiv\Rrightarrow}
\newcommand{\xLleftarrow}[2][]{\ext@arrow 3095\Lleftarrowfill@{#1}{#2}}
\newcommand{\Lleftarrowfill@}{\arrowfill@\Lleftarrow\equiv\equiv}
\makeatother

 \usepackage{enumitem}

\newtheorem{thm}{Theorem}
\newtheorem{prop}[thm]{Proposition}

\newtheorem{coro}[thm]{Corollary}

\theoremstyle{definition}
\newtheorem{defn}[thm]{Definition}
\newtheorem{ex}[thm]{Example}

\theoremstyle{remark}
\newtheorem{rem}[thm]{Remark}
\newtheorem*{rem*}{Remark}

\DeclareDocumentCommand \I { o o } {
    \IfNoValueTF {#1} {
        \IfNoValueTF {#2} {
            \mathtt{I}
        }{
            \mathtt{I}_{#2}
        }
    }{
        \IfNoValueTF {#2} {
            \mathtt{I}(#1)
        }{
            \mathtt{I}_{#2}(#1)
        }
    }
}

\DeclareDocumentCommand \J { o o } {
    \IfNoValueTF {#1} {
        \IfNoValueTF {#2} {
            \mathtt{J}
        }{
            \mathtt{J}_{#2}
        }
    }{
        \IfNoValueTF {#2} {
            \mathtt{J}(#1)
        }{
            \mathtt{J}_{#2}(#1)
        }
    }
}
\DeclareDocumentCommand \Ftr { m o o } {
    \IfNoValueTF {#2} {
        \mathbf{#1}
    }{
        \IfNoValueTF{#2} {
            \mathbf{#1}(#3)
        }{
            \mathbf{#1}_{#2}(#3)
        }
    }
}

\newcommand{\LT}[2]{\cat{LT}_{#1}^{#2}}

\title{Isomorphism invariant metrics}
\date{\today}
\author{P. A. Brooksbank}
\address{Department of Mathematics, Bucknell University, Lewisburg, PA 17837, USA}
\email{pbrooksb@bucknell.edu}

\author{J. F. Maglione}
\address{
Faculty of Mathematics, 
Otto von Guericke
University Magdeburg,
39106 Magdeburg, 
Germany}
\email{joshua.maglione@ovgu.de}
\author{E. A. O'Brien}
\address{Department of Mathematics, University of Auckland, Private Bag 92019,
Auckland, New Zealand}
\email{e.obrien@auckland.ac.nz}
\author{J. B. Wilson}
\address{Department of Mathematics, Colorado State University, Fort Collins, Colorado, 80523-1874, USA}
\email{James.Wilson@ColoState.Edu}

\begin{document}

\begin{abstract}
Within a category $\cat{C}$, having objects $\cat{C}_0$, it may be instructive
to know not only that two objects are non-isomorphic, but also how far from
being isomorphic they are. We introduce pseudo-metrics $d:\cat{C}_0 \times
\cat{C}_0 \to [0,\infty]$ with the property that $x\cong y$ implies $d(x,y)=0$.
We also give a canonical construction that associates to each isomorphism
invariant a pseudo-metric satisfying that condition. This guarantees a large
source of isomorphism invariant pseudo-metrics.  We examine such pseudo-metrics
for invariants in various categories.
\end{abstract}

\maketitle

\section{Introduction}
\label{sec:intro}
Isomorphism invariants are widely used as a means of distinguishing, up to a
suitable notion of equivalence, between structures of various types. The
chromatic number of a graph and the order of a group
are familiar examples of isomorphism invariants. More refined examples include 
the sizes of conjugacy classes
of a group and the fundamental group of a topological space. The utility
of an isomorphism invariant depends, by and large, on two essential 
qualities---how discerning it is 
(how effectively it distinguishes nonequivalent structures),
and how easy it is to compute (either by hand or computer). Heretofore,
the utility of an invariant rested primarily on its ability to 
differentiate---if its value differs on given objects $x$ and $y$, then its job
is done.  
We investigate the question of \emph{how different} $x$ and
$y$ are with respect to a given invariant. We show that each isomorphism invariant
on a category possesses additional information: it 
induces on the objects of the category a universal partial order 
(Theorem~\ref{thm:functor}). This in turn gives rise to a notion of distance between
structures (Corollary~\ref{coro:iso-dist}).

\subsection*{Outline}
Section~\ref{sec:prelims} contains the background we need. 
Section~\ref{sec:categorification} proves our central claim that
an isomorphism invariant can be universally enriched to a functor into a
thin category. Section~\ref{sec:restrict} proposes strategies to extract 
information from this enriched structure of isomorphism invariants. 
Finally, Section~\ref{sec:examples} explores our results applied to
isomorphism invariants used in contemporary research.

\section{Preliminaries}
\label{sec:prelims}

An \emph{isomorphism invariant} is a function $\I$ such that
\begin{align}
    \label{def:iso-inv}
    \textrm{objects } x~\mbox{and}~y~\mbox{are isomorphic} \qquad \Longrightarrow\qquad \I[x]=\I[y].
\end{align}
The precise meaning of 
``isomorphic" varies according to context, but 
we shall use it in the general setting of categories. 
Our notion of equality $\I[x]=\I[y]$ 
is somewhat fluid. For example,
comparing two topologies by their fundamental groups 
requires us to decide isomorphism of groups. This 
is technically an isomorphism problem in its own right, 
but we could simply assume a ``black box" in which 
isomorphic groups are identified as equal. A similar situation occurs
for graphs and modules, which are often compared using 
invariants such as automorphisms and endomorphisms, 
respectively. Even for basic invariants such
as the cardinality of sets, deciding equality is not 
always immediate---indeed writing $|X|=|Y|$ is merely 
a statement that a bijection $X\to Y$ exists. 

Recently a field of study emerged to formalize the
concept of ``identify isomorphisms'' in a manner that does not lead to
paradoxes or issues with structures too large to be sets.
Known as \emph{Univalent Foundations}, 
it builds on nearly a century of refinements to 
Set Theory that began with Bishop Sets, progressed to Martin-L\"of Sets, and achieved its 
current expression through ideas 
from homotopy theory. 
Although the goal was to refine 
concepts of sets, the general study today is known as 
\emph{Homotopy Type Theory}.
While its notation is somewhat different, it 
retains most of the familiar aspects of Set Theory.  We include 
a short summary sufficient for our purposes.

\subsection{Types as generalized sets}\label{sec:types-1}
Formally, types are defined using syntax rules in logic
\cite{Hindley-Seldin}*{Chapters~10-13}. Informally, a \emph{type} is a label
that annotates data. Writing $a:A$ means that $a$ is used only in ways
that all data annotated by $A$ can be used. For example, we write $a:
\mathbb{R}$ to signal that $a$ is used only as a real number.  
Writing $a:A$ is analogous to $a\in A$ in Set Theory, but we dispense with the
supporting axioms of sets. Indeed, all sets are types, but types need not be
sets. Writing $A:\mathfrak{U}$ indicates that $A$ is a type in a universe
$\mathfrak{U}$ of types.

New types can be formed from existing types.  Given types $A$ and $B$, there is
a function type $A\to B$. A {\it term} $f$ of type $A\to B$ is denoted by $f:A\to B$.
As the notation suggests, terms $f$ of type $A\to B$ obey the rule that given
$a:A$, there is an associated term $f(a):B$. The data of a function $f$ in Set
Theory is its graph $\{(a,f(a))\mid a\in A\}\subset A\times B$. In Type Theory,
the data $f:A\to B$ may be a program that transforms $a$ to $f(a)$ (by, for
example, a Turing Machine or $\lambda$-calculus, or an axiom about existence),
so $f:A\to B$ makes sense even if $A$ and $B$ are not sets. An
\emph{isomorphism} between types $A$ and $B$, written $A\cong B$, consists of
functions $f:A\to B$ and $g:B\to A$ whose respective compositions $gf$ and $fg$
reduce as programs to the identity functions $\id_A:A\to A$ and $\id_B:B\to B$. 

In Set Theory, for each ``indexed family'' $\{C_i\mid i\in I\}$ of sets, there
is a Cartesian product $\prod_{i\in I}C_i$ whose elements are tuples $c =
(c_i)_{i\in I}$ with each $c_i\in C_i$. Consider its analogue in Type Theory.
First, take $I:\mathfrak{U}$. Next, define a function $C:I\to \mathfrak{U}$ 
that maps $i:I$ to $C_i:\mathfrak{U}$. Like functions, the terms $c$ of the 
product type $\prod_{i:I}C_i$ take $i:I$ and assign a term  $c_i:C_i$. 

Types are mild conceptual generalizations of na\"ive sets
that are well suited to computation. Indeed, types are 
an essential feature in the forthcoming article
\cite{BDMOBW} that proposes a new computational model 
for algebraic data types
based on higher categories, and in \cite{BMOBW} to 
develop a divide-and-conquer mechanism for isomorphism 
testing in algebra.

\subsection{Equality}
\label{sec:equality}
Equality in Type Theory differs from na\"ive Set Theory in that it is 
based on how data can be used---it is not a judgement 
about sets containing each other's elements.
The concept is 
a version of the Leibniz Law:
\begin{equation}\label{eqn:Leibniz}
    (a= b) \Longrightarrow (\forall P)(P(a)\Leftrightarrow P(b)).
\end{equation}
To define equality 
between terms $a$ and $b$ of a type $A$, a proposition $P(a)$ is replaced 
by a type $X_a:\mathfrak{U}$, where
terms $p:X_a$ witness the validity of $P(a)$. To illustrate, 
suppose $P(m,n)$ is the statement that ``$m\leq n$ for natural numbers $m,n$".
We define a type 
$m\leq_{\N}n$, 
where $p:(m\leq_{\N}n)$ stores the natural number $k$ such that $m+k=n$.
Instead of ``true/false'', we say that $X_{a}$ is \emph{inhabited}
if there is some data (or ``proof'') $p:X_a$, and \emph{uninhabited} 
(analogous to the empty set) if 
it is known to have no data. For example, 
$5\leq_{\N}3$ exists as a type 
but is uninhabited
because there is no $k:\N$ such that $5+k=3$.  

The implication $P(a)\Rightarrow P(b)$ translates to the function type
$X_a\to X_b$.  If $X_b$ is uninhabited and $f:X_a\to X_b$, then $X_a$ must be
uninhabited (akin to $True\not \Rightarrow False$). Thus,
$P(a)\Longleftrightarrow P(b)$ translates to $(f,g) : (X_a\to X_b)\times
(X_b\to X_a)$. 

The Leibniz Law in~\eqref{eqn:Leibniz} now translates to a new
type to capture equality, written $a=_A b$. The data of type $a=_A b$ transforms
to data of type
\[
    \prod_{X:A\to\mathfrak{U}}(X_a\to X_b)\times (X_b\to X_a),
\] 
where $X:A\to \mathfrak{U}$ is a function mapping $a:A$ to $X_a:\mathfrak{U}$.
The type ``$=_A$" is the Martin-L\"of \emph{identity type}; 
see \cite{Hindley-Seldin}*{Chapter~13} for related technical details. 

\subsection{Orders}
\label{sec:orders}
As with equality, an order relation $a\preccurlyeq b$ on a type $A$ 
determines a new type, denoted $\preccurlyeq_A$, where data 
\(
    \mathrm{rel}_{ab}: (a\preccurlyeq_A b)
\)
is interpreted as a ``proof'' of the proposition $P(a,b):\equiv (a\preccurlyeq b)$. 
To see how this works, 
consider a pre-order $\preccurlyeq$, namely 
an order that is both reflexive and transitive.
Each proposition translates to a type: 
the reflexive proposition $R_{\preccurlyeq}:\equiv((\forall
a)(a\preccurlyeq a))$ translates to the type 
$X_{\preccurlyeq}:=\prod_{a:A}(a\preccurlyeq_A a)$, 
and the transitive proposition 
$T_{\preccurlyeq}:\equiv((\forall a,b,c)(a\preccurlyeq b\preccurlyeq c\Rightarrow a\preccurlyeq c))$
translates to the type
\begin{align}
    \label{eq:trans1}
    Z_{\preccurlyeq}:=\prod_{a,b,c:A}\biggl(
        (b\preccurlyeq_A c)\times (a\preccurlyeq_A b)\to (a\preccurlyeq_A c)
        \biggr).
\end{align}
The combination $X_{\preccurlyeq}\times Z_{\preccurlyeq}$  captures 
both reflexive and transitive conditions as one data type, which we denote suggestively 
as $\preccurlyeq_{A}$.  Again $a\preccurlyeq_A b$ is not an assertion 
that $a\preccurlyeq b$.
Rather, it is a data type that, when inhabited by some evidence $e:a\preccurlyeq_A b$, 
allows us to conclude that $a\preccurlyeq b$.
If $a\preccurlyeq_A b$ is shown to be uninhabited, then we 
conclude $a\not\preccurlyeq b$.

A \emph{partial order} is a pre-order $\leq$ that also satisfies the 
antisymmetry law
$S_{\leq}:\equiv ((\forall a,b)(a\leq b\leq a\Rightarrow a=b))$ with associated type
\begin{align}
    \label{eq:trans}
    Y_{\leq}:=\prod_{a,b:A}\biggl(
        (b\leq_A a)\times (a\leq_A b)\to (a=_{A} b)
        \biggr).
\end{align}

\subsection{Univalence}
\label{sec:univalence}
Following the examples of equality and pre-orders, every relation can be turned into a 
type.  For example, we can define a type $A\cong_{\mathfrak{U}}B$ where 
data of this type captures an isomorphism $A\to B$ as defined in 
Section~\ref{sec:types-1}.  
If the types are equal, then 
we have the Martin-L\"of identity type
$A=_{\mathfrak{U}}B$, which means simply that $A$ and $B$ can be used 
interchangeably in our interpretation of 
the Leibniz law. 
Therefore, evidence $e$ of type $(A=_{\mathfrak{U}} B)$ can be used to define the 
identity function $\id_e:A\to B$ which is now an isomorphism of type 
$(A\cong_{\mathfrak{U}}B)$~\cite{HoTT}*{Lemma~2.10.1}.  Call $e\mapsto \id_e$ the 
\emph{identity map}.
The principle of univalence can now be summarized formally by the following axiom.

\begin{description}[leftmargin=0cm]
\item[Univalence axiom] The identity map $(A=_{\mathfrak{U}} B) \longrightarrow
    (A\cong_{\mathfrak{U}} B)$ is invertible.
\end{description}

\textit{Univalence does not make isomorphism obsolete---}on the contrary,
it says that data may be identified \emph{after} an isomorphism has been
exhibited. A comprehensive treatment 
of univalence may be found in~\cite{HoTT}*{Chapter~2}. 

Univalence is germane to isomorphism research in that it facilitates
reductions.  For example, each homeomorphism of pointed topological spaces
$(A,a_0)\cong (B,b_0)$ induces an isomorphism $\pi_1(A,a_0)\cong \pi_1(B,b_0)$
of fundamental groups. Thus, in a univalent framework,
\begin{align*}
    (A,a_0)\cong (B,b_0) 
        \to \pi_1(A,a_0)=_{\mathfrak{U}}\pi_1(B,b_0).
\end{align*}
Likewise, constructions such as automorphism groups and
cardinality of sets can now be regarded as isomorphism 
invariants in a precise sense.

\subsection{Categories}
Our development of categories adapts standard treatments such as
\cite{Jacobson:AlgebraII}*{Chapter~8} to types and univalence, as discussed in
\cite{HoTT}*{Chapter~9}. A \emph{categorical structure} $\cat{C}$ consists of
\emph{objects} 
\(
    \cat{C}_0:\mathfrak{U},
\)
\emph{morphisms}
\(
    \cat{C}_1:\cat{C}_0\times \cat{C}_0\to \mathfrak{U},
\)
\emph{identity morphisms}
\(
    \id:\prod_{x:\cat{C}_0}\cat{C}_1(x,x)
\), 
and a \emph{composition} 
\[
    *:\prod_{x,y,z:\cat{C}_0}\biggl(\cat{C}_1(y,z)\times \cat{C}_1(x,y)\to \cat{C}_1(x,z)\biggr).
\]
We write $f*g$ simply as $fg$. A \emph{pre-category} is a categorical structure
satisfying
\begin{align*}
    &(\forall f:\cat{C}_1(x,y),\forall g:\cat{C}_1(y,z),\forall h:\cat{C}_1(z,w))
    &
    h (g f) & = (h g) f\\
    &(\forall f:\cat{C}_1(x,y))
    &
    \id_y f & = f = f \id_x.
\end{align*}
A morphism $f:\cat{C}_1(x,y)$ is an \emph{isomorphism} if there exists a
morphism $g:\cat{C}_1(y,x)$ such that $\id_x= gf$ and $\id_y= fg$, in which case
we write $x\cong y$.   

Every pre-order $(X,\preccurlyeq)$ is a pre-category
whose morphisms are the terms of type
$x\preccurlyeq_X y$, with composition taken directly from the transitive 
law, and identity maps from the reflexive law.  To
emphasize this perspective, we denote this pre-category by
$\LT{X}{\preccurlyeq}$ to evoke ``less than".  Note that
$\LT{X}{\preccurlyeq}$ is a \emph{thin} pre-category, 
in the sense that there is at most
one morphism between each pair of objects. In fact, \emph{every} thin
pre-category  can be regarded as a pre-order where $x\preccurlyeq y$ if there
is a morphism from $x$ to $y$.  Thus, pre-orders and thin
pre-categories are identical concepts.

A \emph{functor} $\Ftr{F}:\cat{C}\to\cat{D}$ is a homomorphism of
pre-categories, namely a pair of functions $\Ftr{F}_0:\cat{C}_0\to \cat{D}_0$
and $\Ftr{F}_1:\cat{C}_1\to \cat{D}_1$, where
$\Ftr{F}[1][\id_x]=\id_{\Ftr{F}[0][x]}$ and is either \emph{covariant}, so
$\Ftr{F}[1][gf]=\Ftr{F}[1][g]\Ftr{F}[1][f]$, or \emph{contravariant},
so $\Ftr{F}[1][gf]=\Ftr{F}[1][f]\Ftr{F}[1][g]$. To remove notational clutter,
we write $\Ftr{F}(x)$ for $\Ftr{F}[0][x]$ and $\Ftr{F}(f)$ for $\Ftr{F}[1][f]$ when
the context is clear. An \emph{isomorphism} of two pre-categories is a pair
$(\Ftr{F},\Ftr{G})$ of functors between them 
satisfying $\Ftr{FG}=\id$ and $\Ftr{GF}=\id$.

A \emph{natural transformation} $\nu:\Ftr{F}\Rightarrow \Ftr{G}$ between
functors $\Ftr{F},\Ftr{G}:\cat{C}\to \cat{D}$ comprises, for each object $x$ of
$\cat{C}$, a morphism $\nu_{x}:\Ftr{F}(x)\to \Ftr{G}(x)$ in $\cat{D}$ such that, for
every morphism $f:x\to y$ in $\cat{C}$, the following diagram commutes:
\newcommand{\MyJf}{\Ftr{F}(f)}
\newcommand{\MyKf}{\Ftr{G}(f)}
\begin{equation}
    \begin{tikzcd}
             \Ftr{F}(x)
                \arrow[d,"\MyJf"]
                \arrow[r,"\nu_{x}"] 
                &
             \Ftr{G}(x)\arrow[d,"\MyKf"]\\
        \Ftr{F}(y) \arrow[r,"\nu_{y}"] & \Ftr{G}(y)
    \end{tikzcd}
\end{equation}
An \emph{equivalence} of pre-categories is a pair of functors
$(\Ftr{F},\Ftr{G})$ in which there are natural transformations $\id\Rightarrow
\Ftr{FG}\Rightarrow \id$ and $\id\Rightarrow \Ftr{GF}\Rightarrow \id$.

An important factor in our description of pre-categories $\cat{C}$ on objects 
$\cat{C}_0$ is that there is no role in the axioms for the Martin-L\"of identity 
type $=_{\cat{C}_0}$.  This allows us to introduce by definition the 
equality best suited for use of a category in a univalent context.

\begin{description}
\item[Univalence Condition] 
A \emph{univalent category} $\cat{C}$
is a pre-category that satisfies the following:
 for all $x,y:\cat{C}_0$, 
    there is a function 
    $(x\cong_{\cat{C}_0} y) \to (x=_{\cat{C}_0} y)$.
\end{description}

Familiar categories such as groups, topological spaces, and graphs have been
modelled as categories with types in the sense just
described~\cite{HoTT}*{\S\S~6.1 and 9.1}. 
We stress again that we do not ignore 
isomorphisms in these categories. Rather, if an isomorphism 
is presented between
two objects, then 
univalence 
allows us to identify them as equal.  
The role of these isomorphisms is  
to carry out the interchange of elements as required by the Leibniz Law. 

For instance, given a pre-order $\preccurlyeq$ on $\Delta$, 
its associated pre-category $\LT{\Delta}{\preccurlyeq}$ 
induces a category by deciding that every isomorphism $a\cong b$ 
is evidence 
of type $a=_{\Delta} b$.  Notice $a\cong b$ in $\LT{\Delta}{\preccurlyeq}$
means precisely that $a\preccurlyeq b\preccurlyeq a$.
In other words, under this notion of equality the pre-order 
$\preccurlyeq$ becomes a partial order, which we denote by $\leq$.
This is one illustration of a general fact that every pre-category without the univalence condition can be mapped via a weak equivalence functor to a category 
known as the Rezk
completion  \cite{HoTT}*{\S~9.9}. 

\begin{rem}\label{rem:Rezk}
The situation is akin to inducing an equivalence relation $(a\sim
b)\Leftrightarrow (a\preccurlyeq b\preccurlyeq a)$ on $\Delta$ and inducing from
$\preccurlyeq$ a partial order $\leq$ on the resulting partition
$\Delta/_{\sim}$.  However, the univalent formalism allows us to avoid
partitions and newly induced morphisms between equivalence classes of objects.
\end{rem}

\begin{rem}
In category theory based on other foundations---such as 
von Neumann--Bernays--G\"odel Set-Class Theory, or Grothendieck 
universes---the notions of pre-category and category coincide. Such 
foundations admit notions similar to univalence by taking a 
\emph{skeleton} of the category: using a suitably strong version of the 
Axiom of Choice, we choose a representative from each 
isomorphism type in the category.  
Skeletons essentially resolve isomorphism 
questions by invoking an axiom---hence, for the purpose of isomorphism 
research, we avoid them.
\end{rem}

\section{Metrics from Isomorphism Invariants}
\label{sec:categorification}
Categories provide a suitably general framework within which to formulate the
problems we discuss, and to present the main outcomes of our study.  
We now explain how each isomorphism invariant on a category 
gives rise to a universal partial order, which in turn leads 
to a notion of distance between the objects of the category. 

Fix a univalent category $\cat{C}$ with objects $\cat{C}_0$.
A \emph{pseudo-metric} on $\cat{C}$ is a function
$d:\cat{C}_0\times \cat{C}_0\to [0,\infty]$ such that for all
$x,y,z:\cat{C}_0$ the following 
conditions hold:
\begin{enumerate}\item[(i)] $d(x,x)=0$, 
\item[(ii)] $d(x,y)=d(y,x)$, and
\item[(iii)] $d(x,y)\leq d(x,z)+d(z,y)$.
\end{enumerate}
A \emph{metric} is a pseudo-metric in which $x=y$ if $d(x,y)=0$. Each
pseudo-metric induces a metric simply by identifying points $x,y$ for which
$d(x,y)=0$ and working instead in the quotient metric topology. 
The condition $d(x,y) < \infty$ is
often imposed for metrics, and it can be attained by canonically decomposing
$(\cat{C}_0, d)$ into subspaces carrying only finite 
metrics~\cite[Ch.~1]{BBI:metric-geometry}. 

Conditions (i)--(iii) are standard requirements for any traditional notion of
distance. Condition (i) is 
equivalently written $d(x,y)=0$ if $x=y$, so in our
univalent setting 
\[
    x\cong y \quad \Longrightarrow \quad x=y\quad \Longrightarrow \quad d(x,y)=0.
\]
Thus, a pseudo-metric in our context treats isomorphic objects as having a
distance of $0$.  In this sense, distances greater than $0$ are a measure of
non-isomorphism.  

\subsection{Characteristic metrics of isomorphism invariants}
Isomorphism invariants and pseudo-metrics are closely related. Indeed, each
isomorphism invariant $\I:\cat{C}_0\to \Delta$ from a category $\cat{C}$ into some
type $\Delta$ has a \emph{characteristic} pseudo-metric
$c_{\I}:\cat{C}_0\times\cat{C}_0\to 
[0,\infty]$,
where for objects $x,y$,
\[
    c_{\I}(x,y)  = \begin{cases} 
        0 & \text{ if } \I[x]=\I[y] \\ 1 & \text{ otherwise. } 
    \end{cases}
\]

Conversely, suppose we are given a pseudo-metric 
$d:\cat{C}_0\times \cat{C}_0\to [0,\infty]$.
Let $[0,\infty]^{\cat{C}_0}$ denote functions $\cat{C}_0\to [0,\infty]$.
For each $x:\cat{C}_0$, define $\I_d(x):\cat{C}_0\to [0,\infty]$, where 
\(
    \I_d(x)(z)=d(x,z).
\)
If $x\cong y$ then, for all objects $z$,
\begin{align*}
    \I_d(x)(z) = d(x,z)=d(y,x)+d(x,z)\geq d(y,z)=\I_d(y)(z).
\end{align*}
Likewise, 
\(
    \I_d(y)(z) = d(x,y)+d(y,z)\geq d(x,z)=\I_d(x)(z),
\)
so $\I_d(x)=\I_d(y)$ in $[0,\infty]^{\cat{C}_0}$. 
Hence, $\I_d:\cat{C}_0\to [0,\infty]^{\cat{C}_0}$ sending
$x\mapsto \I_d(x)$ is an isomorphism invariant.

If $\I:\cat{C}_0\to \Delta$ is an isomorphism invariant, and $x$ is an
object, then 
\begin{align*}
\I_{c_{\I}}(x)(y)= \begin{cases}
    0 & \text{ if } \I[x]=\I[y] \\ 1 & \text{ otherwise.}
    \end{cases}
\end{align*}
so $\I[x]$ can be recovered from $\I_{c_{\I}}(x)$. 
On the other hand, if
$d:\cat{C}_0\times\cat{C}_0\to [0,\infty]$ is a pseudo-metric, 
then $\I_d(x) = \I_d(y)$ implies $d(x,z)=d(y,z)$ for each object $z$---in
particular, if $z=y$, then $d(x,y)=d(y,y)=0$. Hence 
\begin{align*}
    c_{\I_d}(x,y)= \begin{cases} 
        0 & \text{ if } d(x,y)=0 \\ 1 & \text{ otherwise} 
    \end{cases}
\end{align*}
and the value of $d(x,y)$ is lost. Thus, $\I\mapsto c_{\I}$ and $d\mapsto \I_d$
are only partial inverses.

In summary, the characteristic metric possesses no more information than 
the isomorphism invariant. However, the main result in the 
next section shows that each isomorphism invariant has a more discerning 
metric from which useful information can often be extracted.

\subsection{Main Theorem}
The following result asserts that isomorphism invariants correspond to functors
into thin categories. The proof uses the notation for orders 
introduced in Section~\ref{sec:orders}.

\begin{thm}\label{thm:functor}
    Let $\cat{C}$ be a category. Let $\I:\cat{C}_0\to
    \Delta$ be a surjective isomorphism invariant from the objects of $\cat{C}$ 
    onto a type $\Delta$. 
There is a partial order
    $\leq$ on $\Delta$ and a functor $\Ftr{J}:\cat{C}\to \LT{\Delta}{\leq}$
    such that $\Ftr{J}[][x]=\I[x]$ for each $x : \cat{C}_0$. Furthermore,
    $(\LT{\Delta}{\leq},\Ftr{J})$ is universal in the following sense:
    \smallskip

    \begin{minipage}{0.9\textwidth}
    \noindent For every thin category $\cat{R}$ with objects $\Delta$,
    and every functor $\Ftr{K}:\cat{C}\to \cat{R}$, where
    $\Ftr{K}[][x]=\I[x]$ for each $x : \cat{C}_0$, there is a unique functor 
    $\Ftr{L}:\LT{\Delta}{\leq} \to \cat{R}$ 
    and a natural transformation from $\Ftr{L} \Ftr{J}$ to $\Ftr{K}$.
    \end{minipage}
\end{thm}

Recall 
that equality between objects in a category is determined by isomorphisms 
in the category.
In particular, equality between terms of type $\Delta$ in
Theorem~\ref{thm:functor} is determined by isomorphisms in
$\LT{\Delta}{\leq}$.  

\begin{proof}
    For terms $\delta,\epsilon$ of $\Delta$, write
    $\delta\preccurlyeq\epsilon$ if there is a sequence $(f_i:x_i\to y_i \mid
    1\leqslant i\leqslant n)$ of morphisms in $\cat{C}$ such that
    $\I[y_i]=\I[x_{i+1}]$ for each $1\leqslant i \leq n-1$, with
    $\delta=\I[x_1]$ and $\epsilon=\I[y_n]$. 
    
    Fix $\delta:\Delta$. As
    $\I$ is surjective, there is an object $x$ of $C$ with $\I[x]=\delta$. 
Hence $\delta\preceq\delta$ using the sequence $(\id_x)$, so
    $\preceq$ is reflexive. Also, if $\delta\preccurlyeq \epsilon$ via $(f_i:x_i\to
    y_i \mid 1\leqslant i\leqslant n)$ and $\epsilon\preccurlyeq \phi$ via
    $(g_i:z_i\to w_i \mid 1\leqslant i\leqslant m)$, then $\delta=\I[x_1]$,
    $\I[y_n]=\epsilon=\I[z_1]$, and $\I[w_m]=\phi$. Hence, the sequence $(f_1,
    \ldots, f_n, g_1, \ldots, g_m)$ shows that $\epsilon\preccurlyeq\phi$, so
    $\preceq$ is transitive. Thus, $\preccurlyeq$ is a pre-order.

    Next we define a functor $\Ftr{J}:\cat{C}\to \LT{\Delta}{\preccurlyeq}$.
    For each $x:\cat{C}_0$, we require $\Ftr{J}(x)=\I[x]$. If $f:x\to
    y$ is a morphism in $\cat{C}$ then, by definition, $\I[x]\preccurlyeq \I[y]$
    so we define $\Ftr{J}(f)=\mathrm{rel}_{\I[x]\,\I[y]}$. In particular,
    $\Ftr{J}(\id_x)=\mathrm{rel}_{\I[x]\,\I[x]}$. If $f:x\to y$ and $g:y\to z$
    are morphisms, then 
    \[\Ftr{J}(g)\Ftr{J}(f)=
    (\mathrm{rel}_{\I[y]\,\I[z]})(\mathrm{rel}_{\I[x]\,\I[y]}) =
    \mathrm{rel}_{\I[x]\,\I[z]}=\Ftr{J}(gf),
    \] 
    so $\Ftr{J}$ is a functor.
    
    It remains to establish the stated universal property. Suppose 
    $\Ftr{K}:\cat{C}\to \cat{R}$ is a functor, where $\cat{R}$ is a thin
    pre-category with objects $\Delta$, and $\Ftr{K}$ agrees with $\I$ on objects
    in $\cat{C}$. We define a functor $\Ftr{L}: \LT{\Delta}{\preccurlyeq}\to
    \cat{R}$ behaving as claimed. For each object $\delta$ of
    $\LT{\Delta}{\preccurlyeq}$, put $\Ftr{L}(\delta)=\delta$. Let
    $\mathrm{rel}_{\delta\epsilon}$ be a morphism of
    $\LT{\Delta}{\preccurlyeq}$, so $\delta\preccurlyeq \epsilon$ via some
    sequence $(f_i:x_i\to y_i \mid 1\leqslant i\leqslant n)$. Since $f_i$ is a
    morphism in $\cat{C}$ and $\Ftr{K}$ is a functor, $\Ftr{K}(f_i) : \I[x_i]
    \to \I[y_i]$ for each $i$. As $\I[y_i] = \I[x_{i+1}]$ for each $1\leq i\leq
    n-1$, the morphisms $\Ftr{K}(f_i)$ are composable, and $\Ftr{K}(f_n)\cdots
    \Ftr{K}(f_1):\delta\to \epsilon$. As $\cat{R}$ is thin, this is the unique
    morphism in $\cat{R}$ with domain $\delta$ and codomain $\epsilon$, so we
    define $\Ftr{L}(\mathrm{rel}_{\delta\epsilon}) = \Ftr{K}(f_n)\cdots
    \Ftr{K}(f_1)$. 
Hence $\Ftr{L}$ is a functor
    $\LT{\Delta}{\preccurlyeq}\to \cat{R}$, so $\Ftr{K}(x)=\Ftr{LJ}(x)$ and
    $\Ftr{K}(f)=\Ftr{LJ}(f)$, as required.

    Finally, as in Remark~\ref{rem:Rezk},
    let $\Ftr{J}':\LT{\Delta}{\preccurlyeq}\to \LT{\Delta}{\leq}$ be the functor
    taking the pre-category of the pre-order $\preccurlyeq$ 
    to the associated category with partial order $\leq$. If $\cat{R}$ is a category,
    then the above construction factors through $\Ftr{J}'$. This completes the
    proof.
\end{proof}

\subsection{From orders to distances}
We now show that the partial order associated with $\I$ in
Theorem~\ref{thm:functor} leads to a more discerning metric. Our treatment of
metrics on partially ordered sets follows~\cite{Monjardet:metric-posets}.

Let $\leq$ be a partial order for a type $P$.  If $p\leq q$ and $p\neq q$, then $p<q$. 
If $p<r\leq q$ implies $q=r$, then $q$ \emph{covers} $p$. 
The \emph{Hasse
diagram} of $(P,\leq)$ is the directed graph with vertices $P$ and edges
\(
    \mathcal{E}  = \{ (p,q) \mid q \text{ covers }p\}.
\)
The \emph{cover graph} $\mathcal{C}(P)$ is the undirected graph generated by the
Hasse diagram. 
Define the \emph{graph metric} $d_{\leq} : P \times P \to [0, \infty]$ such that
$d_{\leq}(p,q)$ is either the minimal length of all finite paths in
$\mathcal{C}(P)$ between $p$ and $q$ or infinite if none exist.

\begin{coro}\label{coro:iso-dist} 
Let $\cat{C}$ be a category. Let $\I:\cat{C}_0\to
    \Delta$ be a surjective isomorphism invariant from the objects of $\cat{C}$ 
    onto a type $\Delta$. There exists a partial order $\leq$ on
    $\Delta$, and $d_{\I} : \cat{C}_0\times \cat{C}_0 \to [0,\infty]$ given by $(x,y)\mapsto d_{\leq}(\I[x], \I[y])$ is a pseudo-metric.
\end{coro}

\section{A User's Guide}
\label{sec:restrict}
The constructions in
Theorem~\ref{thm:functor} and Corollary~\ref{coro:iso-dist} 
are general---they may be applied to any 
isomorphism invariant on any category.
However, doing 
so indiscriminately can lead to underwhelming results.  
To extract meaningful information, we must often impose 
restrictions on the categories, such as by restricting to 
either the epimorphisms or the monomorphisms
in the category.  This is analogous to contemporary 
strategies in isomorphism testing that study objects as 
quotients of larger structures. In this section 
we reveal various 
impediments to using 
our isomorphism invariant metrics, and discuss some ways 
to overcome them.

\subsection{Cycles}
Consider the 
category $\cat{Vect}$ of finite-dimensional vector spaces, 
whose morphisms are linear maps. 
The dimension function $\dim:\cat{Vect}_0\to \N$
is an isomorphism invariant. 
By Theorem~\ref{thm:functor}, there is an associated 
partial order $\sqsubseteq$ on $\N$.  
However, this is not the usual successor order we might expect. 
If $U$ and $V$ are vector spaces, then there are linear maps 
$U\to 0\to V$. Applying the functor $\Ftr{J}$  
into $\LT{\N}{\sqsubseteq}$
shows that
\(
    \dim U \sqsubseteq 0\sqsubseteq \dim V,
\)
so in this category all objects are isomorphic.  Since $\LT{\N}{\sqsubseteq}$ is univalent, all objects
in this category are equal.  In other words,  as in Remark~\ref{rem:Rezk},
$(\N, \sqsubseteq)$ is
defined on the trivial partition of $\N/_{\sim}$ where all natural
numbers are equal. 
This is a common situation in categories rich with morphisms, and it
arises due the following phenomenon. 

\begin{defn}
    A \emph{cycle} in a category $\cat{C}$ is a sequence $(f_0, \dots, f_{n-1})$ of
    morphisms $f_i : x_i \to x_{i+1}$ for all $i\in \Z/n\Z$, where $n\geq 2$. 
    A cycle is \emph{trivial} if $x_0\cong\cdots\cong x_{n-1}$.
\end{defn}
    
For example, $0\to \mathbb{Z}\to 0$ is a cycle in the category of abelian groups 
but $\id:\mathbb{Z}\to \mathbb{Z}$ and $\mathbb{Z}\to \mathbb{Z}/2\to 0$ are not.
If $C$ has a cycle, then a functor 
from $C$ 
into a thin 
category is forced to assign all terms in the cycle to a common point.
For non-trivial cycles this may result in unwanted identifications.

\subsection{Restricting morphisms}
\label{sec:break-cycles}
Cycles are often broken by removing some of the morphisms that lead to them.
For example, the cycle $0\to \mathbb{Z}\to 0$ in the category of abelian groups
consists of an injective function followed by a surjective function. Restricting
to just one class of function---the injective ones, say---eliminates one of the
edges, thereby breaking the cycle.

\begin{defn}
    Let $\cat{C}$ be a category. 
    
    A morphism $f$ in $\cat{C}$ 
    is a \emph{monomorphism} if, 
    for all composable morphisms $g,h$ in $\cat{C}$, 
    $fg=fh$ implies $g=h$. 
Let $\mono{C}$ be the subcategory 
    with the same objects as $\cat{C}$ whose morphisms 
    are the monomorphisms in $\cat{C}$.
    
    Similarly, $f$ is an \emph{epimorphism} if 
    $gf=hf$ implies $g=h$. 
Let $\epi{C}$ be the subcategory 
    with the same objects as $\cat{C}$ whose morphisms 
    are the epimorphisms in~$\cat{C}$.
\end{defn}

We apply Theorem~\ref{thm:functor} to 
the invariant $\dim:\;\mono{Vect}_0\to \N$.
Let $\leq$ denote the usual successor order on $\N$:
\[
    m\leq n \qquad \Longleftrightarrow \qquad (\exists k)(m+k=n).
\]

If $f:U\hookrightarrow V$, with $m=\dim U$ and $n=\dim V$,
then $m\leq n$ because every basis for $f(U)$ can 
be extended to a basis for $V$. Thus, 
$\Ftr{J}(f)=\mathsf{rel}_{mn}$
is the unique  morphism of type $m\leq_{\N} n$.
Moreover,  whenever
$m\leq n$, there is a vector space
$U$ with $m=\dim U$, a vector space $V$ with $n=\dim V$, 
and a monomorphism $f:U\to V$, so no relations can be dropped.
Hence, if  $\cat{P}$ is a thin category with $\cat{P}_0=\N$,
and $\Ftr{K}:\;\mono{Vect}\to \cat{P}$ is a functor
such that $\Ftr{K}(V)=\dim V$, then $\Ftr{K}(f)$ is 
the unique morphism of $\cat{P}$ with domain 
$m$ and codomain $n$. 
Thus, $\LT{\N}{\leq}$ maps into $\cat{P}$ via 
the unique function $\Ftr{J}(f)=\mathsf{rel}_{mn}\mapsto \Ftr{K}(f)$.
Finally, the pseudo-metric 
induced by the dimension invariant is, from Corollary~\ref{coro:iso-dist},
\[
    d_{\dim}(U,V) = |\dim U-\dim V|.
\]

\subsection{Cantor--Schr\"oder--Bernstein categories.}
\label{sec:CSB}
We cannot always break cycles just by restricting morphisms. 
Consider, for example, the free group $F_n$ on $n$ letters.  
For every $n>2$, by just using words in two of the $n$ variables,
we get a monomorphism $F_2\hookrightarrow F_n$.
A result of Dehn shows that 
there is also a monomorphism $F_n\hookrightarrow F_2$, so 
there are non-trivial cycles 
$F_2\hookrightarrow F_n\hookrightarrow F_2$
in the category of groups
whose morphisms are monomorphisms.
For {\it concrete} categories 
(those with a faithful functor to $\cat{Set}$)
this can only occur for infinite sets.
This suggests that we impose ``set-like'' conditions on our 
category to avoid cycles.

\begin{defn}
    A category $\cat{C}$ is a 
    \emph{monic Cantor--Schr\"oder--Bernstein (CSB) category} 
    if its monomorphisms satisfy the following condition:
    \[
        \mbox{if}~f: x \hookrightarrow y~\mbox{and}~g : y\hookrightarrow x,~ 
        \mbox{then}~x\cong y. 
    \]
    Epic CSB categories satisfy an analogous condition on 
    epimorphisms:
    \[
        \mbox{if}~f: x \twoheadrightarrow  y~\mbox{and}~g : y\twoheadrightarrow  x,~ 
        \mbox{then}~x\cong y. 
    \]
\end{defn}
   
Many familiar concrete categories---such as finite groups and 
finite graphs---are CSB categories, as are larger categories (subject to
an appropriately strong Axiom of Choice). The following 
elementary result demonstrates the importance of CSB categories 
to Theorem~\ref{thm:functor}.

\begin{prop}
    In a monic CSB category $\cat{C}$, all cycles in $\mono{C}$ are trivial. 
    In an epic CSB category $\cat{C}$, all cycles in $\epi{C}$ are trivial. 
\end{prop}

\begin{proof} 
    Consider a cycle $(f_0, \dots, f_{n-1})$ of morphisms in $\mono{C}$ 
    for a monic CSB category $\cat{C}$.
    Then $f_{n-2}\cdots f_1f_0 : x_0\to x_{n-1}$ and $f_{n-1} : x_{n-1} \to
    x_0$ are both monomorphisms. By the CSB property, this implies that 
    $x_{n-1}\cong x_0$. The result now follows by induction. The proof 
    for epic CSB categories is almost identical. 
\end{proof}

\subsection{Virtual cycles}
There are other potential obstacles to defining useful distances and not all of
them can be removed by restricting morphisms. Returning to fundamental groups,
consider the following continuous surjective maps of pointed topological spaces:
\begin{center}
    \begin{tikzcd}
        (\mathbb{R}^1,0) \arrow[r,"e^{i\theta}"]
        & (S^1,1) \arrow[r]
        & (\mathrm{pt},\mathrm{pt}) .
    \end{tikzcd}
\end{center}
Applying $\pi_1$, we obtain the group homomorphisms $1\to \mathbb{Z}\to 1$. Hence,
$1\leq \mathbb{Z}\leq 1$ and it follows that $1=\mathbb{Z}$.  
This phenomenon occurs with every universal covering space. 

\begin{defn}
    Let $\cat{C}$ be a category, and let $\I$ be an isomorphism invariant 
    on $\cat{C}$.
    A \emph{virtual $\I$-cycle} is a sequence $(f_0, \dots, f_{n-1})$ of
    morphisms $f_i : x_i \to y_{i}$ where $n\geq 2$ and $\I[x_{i+1}]=\I[y_i]$ 
    for all $i\in \Z/n\Z$. 
    A virtual $\I$-cycle is \emph{trivial} if $\I[x_0]=\cdots=\I[x_{n-1}]$.
\end{defn}

Taking the fundamental group as the invariant,
$\pi_1(\mathbb{R}^1,0)=1=\pi_1(\mathrm{pt},\mathrm{pt})$ in our earlier
illustration, whereas $\pi_1(S^1,1)\not\cong 1$. Thus, we have a nontrivial
virtual cycle. Sometimes these can be disrupted.

\begin{prop}\label{prop:new-iso-inv}
    Suppose that $\cat{C}$ and $\cat{D}$ are 
categories with a functor $\Ftr{G}:\cat{C}\to
    \cat{D}$. If $\I_1$ and $\I_2$ are isomorphism invariants on $\cat{C}$ and
    $\cat{D}$ respectively, then 
    \[
        (\I_1\cup\I_2)(x) := (\I_1(x),\I_2(\Ftr{G}(x)))
    \] 
    is an isomorphism invariant of $\cat{C}$.  
\end{prop}
This observation can be used to break the virtual cycle in our 
fundamental group example. Let both $\cat{C}$ 
and $\cat{D}$ be the category of pointed topological spaces, 
and let $\Ftr{G}$ be the identity functor.  
Let $\I_1(X,x_0)=\pi_1(X,x_0)$ and $\I_2(X,x_0)=\dim X$,
the topological dimension of $X$. The 
invariant $\I_1\cup \I_2$ breaks the virtual cycle:
\begin{align*}
    (\I_1\cup\I_2)(\mathbb{R}^1,0)=(0,1)
    \geq (\I_1\cup \I_2)(S^1,0)=(\mathbb{Z},1)\geq (\I_1\cup\I_2)(\mathrm{pt},\mathrm{pt})=(0,0).
\end{align*}

\begin{defn}
    A category $\cat{C}$ is a \emph{pigeonhole} category if $\cat{C}$ is
    concrete 
and every morphism $f:x\to y$ in $\cat{C}$ satisfying $|x|=|y|$ is an isomorphism. 
\end{defn}
\noindent
Both $\mono{FinSet}$ and $\epi{FinSet}$ are pigeonhole categories.

\begin{prop}
    Let $|\cdot|$ be the cardinality invariant on $\cat{Set}$. If $\I$ is an
    isomorphism invariant on a pigeonhole category $\cat{C}$, then all virtual
    $(\I\cup |\cdot|)$-cycles are trivial.
\end{prop}
\begin{proof}
    Since $\cat{C}$ is concrete, $\tilde{\I}=\I\cup|\cdot|$ is an isomorphism
    invariant on $\cat{C}$ by Proposition~\ref{prop:new-iso-inv}. Consider the
    virtual $\tilde{\I}$-cycle $(f_0,\ldots,f_{n-1})$, where $f_i:x_i\to y_i$
    for $0\leq i\leq n-1$. Hence $\tilde{\I}(x_0)\leq
    \tilde{\I}(y_0)=\tilde{\I}(x_1)\leq  \cdots \leq
    \tilde{\I}(y_{n-1})=\tilde{\I}(x_0)$. Therefore 
    \begin{align*}
        (\I[x_0],|x_0|)\leq
        (\I[y_0],|y_0|)=(\I[x_1],|x_1|)\leq
        \cdots \leq 
        (\I[y_{n-1}],|y_{n-1}|)=
        (\I[x_{0}],|x_{0}|).
    \end{align*}
    It follows that 
    \[
        |x_0| \leq |y_0|=|x_1|\leq |y_1|=|x_2|\leq \cdots \leq |y_{n-1}|=|x_0|.
    \]
Since $\cat{C}$ is a pigeonhole category, each $f_i$ is an isomorphism, so
    $\I[x_i]=\I[y_j]$ for all $0\leq i,j\leq n-1$. Thus,
    $\tilde{\I}(x_i)=\tilde{\I}(y_j)$, and the virtual cycle is trivial.
\end{proof}

\section{Examples}
\label{sec:examples} 
In the previous section we discussed general phenomena that 
can impede a meaningful attribution of distance to an isomorphism 
invariant. Even in situations where meaningful pseudo-metrics exist,   
there remains the task of identifying explicitly the 
universal partial order whose existence is guaranteed by 
Theorem~\ref{thm:functor}. 
Its feasibility 
depends both on the nature of the category and the subtlety 
of the invariant. Through a series of vignettes of increasing 
complexity, we 
illustrate some approaches to the identification question. 

\subsection{Chromatic number of graphs}
For $\mathcal{G}$ in the category $\cat{Grph}$ of
finite, simple graphs, the chromatic number $\chi(\mathcal{G})$ is the
least number of colors needed to color the vertices of $\mathcal{G}$ such that
no two adjacent vertices have the same color. This is an isomorphism invariant
on $\cat{Grph}$ taking values in $\Z^+=\{1,2,\dots\}$.  

If $\mathcal{H}\hookrightarrow \mathcal{G}$ is an injection of graphs which
induces injections on the vertex and edge sets, then a coloring of $\mathcal{G}$
induces a coloring of $\mathcal{H}$. Hence, $\chi(\mathcal{H})\leq
\chi(\mathcal{G})$ under the successor ordering $\leq$ of $\Z^+$. The
chromatic number of the complete graph $K_n$ on $n$ vertices is $n$.  For each
$1\leq i\leq n$ there is an injection $K_i\hookrightarrow K_n$ obtained from the
first $i$ vertices in $K_n$ and its incident edges.  
Hence, the chromatic number invariant is surjective, and every relation $m\leq
n$ is covered by some monomorphism. By Theorem~\ref{thm:functor},
$(\Z^+,\leq)$ is the universal order for the chromatic number invariant on
$\mono{Grph}$. The corresponding distance between graphs $\mathcal{G}$ and
$\mathcal{H}$ is given by
\begin{align*}
    d_{\chi}(\mathcal{G},\mathcal{H})=|\chi(\mathcal{G})-\chi(\mathcal{H})|.
\end{align*}
The graphs $\mathcal{H}$ and $\mathcal{J}$ in Figure~\ref{fig:chromatic}
both have chromatic number 3, so $d_{\chi}(\mathcal{J},\mathcal{H})=0$.
Also $\chi(\mathcal{K})=4$, so $d_{\chi}(\mathcal{J},\mathcal{K})=1 =
d_{\chi}(\mathcal{H},\mathcal{K})$, which reinforces the perception that the
chromatic number is a coarse invariant.

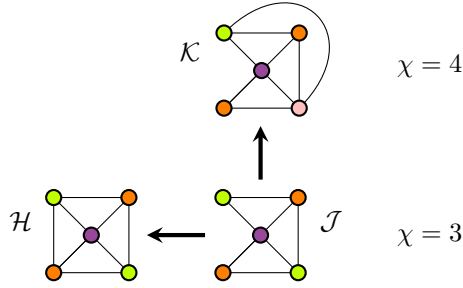
\begin{figure}[h]
    \vspace{-2em}
    \begin{center}
        \begin{tikzpicture}[scale=0.75]
        \node (An) at (-3.25,0.25) {$\mathcal{H}$};
        \node (A) at (-2,0) {\begin{tikzpicture}[scale=0.5]
            \coordinate (a) at ( 0, 0);
            \coordinate (b) at (-1, 1);
            \coordinate (c) at ( 1, 1);
            \coordinate (d) at ( 1,-1);
            \coordinate (e) at (-1,-1);

            \draw (a) -- (b) -- (c) -- (d) -- (e) -- cycle;
            \draw (a) -- (c);
            \draw (a) -- (d);
            \draw (a) -- (e);
            \draw (b) -- (e);

            \draw[thick,fill=Purple] (a) circle (0.2);
            \draw[thick,fill=lime] (b) circle (0.2);
            \draw[thick,fill=orange] (c) circle (0.2);
            \draw[thick,fill=lime] (d) circle (0.2);
            \draw[thick,fill=orange] (e) circle (0.2);
        \end{tikzpicture}};

        \node (Bn) at (-0.25, 3.25) {$\mathcal{K}$};
        \node (B) at (1, 2.1) {};
        \node (Bgraph) at (1.6,3.5) {\begin{tikzpicture}[scale=0.5]
            \coordinate (a) at ( 0, 0);
            \coordinate (b) at (-1, 1);
            \coordinate (c) at ( 1, 1);
            \coordinate (d) at ( 1,-1);
            \coordinate (e) at (-1,-1);

            \draw (a) -- (b) -- (c) -- (d) -- (e) -- cycle;
            \draw (a) -- (c);
            \draw (a) -- (d);
            \draw (a) -- (e);
            \draw (b) edge[out=45, in=45, looseness=2.5] (d);

            \draw[thick,fill=Purple] (a) circle (0.2);
            \draw[thick,fill=lime] (b) circle (0.2);
            \draw[thick,fill=orange] (c) circle (0.2);
            \draw[thick,fill=pink] (d) circle (0.2);
            \draw[thick,fill=orange] (e) circle (0.2);
        \end{tikzpicture}};

        \node (Cn) at (2.25,0.25) {$\mathcal{J}$};
        \node (C) at (1,0) {\begin{tikzpicture}[scale=0.5]
            \coordinate (a) at ( 0, 0);
            \coordinate (b) at (-1, 1);
            \coordinate (c) at ( 1, 1);
            \coordinate (d) at ( 1,-1);
            \coordinate (e) at (-1,-1);

            \draw (a) -- (b) -- (c) -- (d) -- (e) -- cycle;
            \draw (a) -- (c);
            \draw (a) -- (d);
            \draw (a) -- (e);

            \draw[thick,fill=Purple] (a) circle (0.2);
            \draw[thick,fill=lime] (b) circle (0.2);
            \draw[thick,fill=orange] (c) circle (0.2);
            \draw[thick,fill=lime] (d) circle (0.2);
            \draw[thick,fill=orange] (e) circle (0.2);
            
        \end{tikzpicture}};

        \draw[ultra thick,-stealth]  (C)-- (A);
        \draw[ultra thick,-stealth]  (C)-- (B);

        \node (F) at (4,3) {$\chi=4$};
        \node (D) at (4,0) {$\chi=3$};
    \end{tikzpicture}
    \end{center}
    \vspace{-0.75em}
    \caption{Illustration of chromatic distance}\label{fig:chromatic}
\end{figure}

\subsection{Order of a finite group}
The successor order on $\Z^+$ is not the universal
order for every isomorphism invariant into $\Delta=\Z^+$.  Let $\mono{FinGrp}$ be the category whose
objects are finite groups and whose morphisms are group monomorphisms. 
Consider the partial order and pseudo-metric generated by the order of a group, namely the
function $|\cdot|:\,\mono{FinGrp}\ \to \Z^+$ sending $G\mapsto |G|$. As
monomorphisms in this category are injections of sets, it follows
that $H\hookrightarrow G$ induces $|H|\leq |G|$, but this order is not universal.  
The universal partial order is inferred from Lagrange's Theorem:
\[
    m|n \Longleftrightarrow (\exists k)(mk=n).
\]
This is indeed equivalent to the partial order of Theorem~\ref{thm:functor}, and
the induced pseudo-metric on $\mono{FinGrp}$ is given by
\begin{align*}
    d(G,H) & = \sum_{p\text{ prime}}|\nu_p(|G|)-\nu_p(|H|)|,
    &\text{where } n & = \prod_p p^{\nu_p(n)}.
\end{align*}
In particular, each group of order $n$ is closer to each of its Sylow subgroups
than it is to any group of order $n\pm 1$.

\subsection{Abelian invariants} \label{sec:abelian-inv}
Just like the chromatic number of a graph, 
the order of a finite group is a very coarse invariant. 
For example, there are
almost 50 billion non-isomorphic groups of order 1024 \cite{millenium}. If we
restrict attention to finite \emph{abelian} groups, then 
a little more sophistication 
yields 
a \emph{defining} invariant for such groups---one that
settles the isomorphism question.

Let $\Delta$ be the set of \emph{divisor chains} $(d_1,\ldots,d_n)$, where
$d_i\in\N$ and $d_1\mid \cdots \mid d_n$. There is a function $\mathrm{div}$
from the objects of the category $\cat{FinAb}$ of finite abelian groups to
$\Delta$, where $\mathrm{div}\,A= (d_1,\ldots,d_n)$ means that the abelian group
$A\cong \mathbb{Z}/{d_1}\oplus \cdots \oplus \mathbb{Z}/{d_n}$. In this
sense $\mathrm{div}$ is a defining isomorphism invariant for $\cat{FinAb}$.
The terms in $\mathrm{div}\,A$ are the \emph{abelian invariants} of $A$.

Applying Theorem~\ref{thm:functor} to $\epi{FinAb}$ gives the following partial
order on $\Delta$:
\begin{center}
$(d_1,\dots,d_n) \leq  (e_1,\dots,e_m)$ 
if, and only if, $m\leq n$ and  
$e_i\mid d_{n-m+i}$ for all $i$.
\end{center}

Consider $A_1 = \mathbb{Z}/2\oplus\mathbb{Z}/2\oplus
    \mathbb{Z}/4$, and $A_2 = \mathbb{Z}/2\oplus\mathbb{Z}/2\oplus
    \mathbb{Z}/2$, and $A_3 = \mathbb{Z}/8$, so 
    $\mathrm{div}\,A_1=(2,2,4)$, $\mathrm{div}\,A_2=(2,2,2)$, 
    and $\mathrm{div}\,A_3=(8)$. 
The distances between these groups can be seen in 
Figure~\ref{fig:div-chain}, but we briefly examine the reasons.
Evidently, the pairwise distance between the groups is at least 1.
As $A_1$ maps onto $A_2$ with a kernel of order $2$, 
it follows that $(2,2,4)$ is covered by $(2,2,2)$, so 
\(
    d(A_1, A_2) = 1.
\)
There is no surjection from $A_1$ and $A_3$, or 
conversely, so $d(A_1, A_3)
\geq 2$. Thus $A_1$ is closer to $A_2$ than it is 
to $A_3$. In fact, $(2,2,4) < (2,4) < (4)$ and $(8) < (4)$. 
The largest image of $A_3$ in $A_1$ is isomorphic 
to $\mathbb{Z}/4$, so
\(
    d(A_1, A_3) = d(A_1, \mathbb{Z}/4) + 1 = 3.
\) 

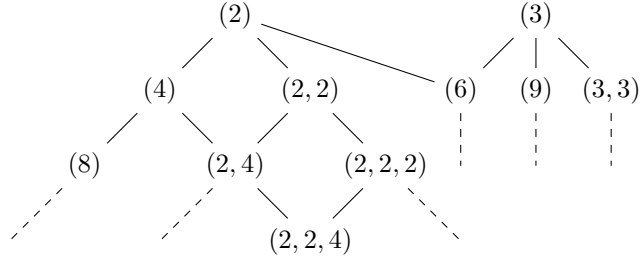
\begin{figure}[htbp]
\begin{center}
    \begin{tikzpicture}
        \node (224) at (1,-3) {$(2,2,4)$};
        \node (222) at (2,-2) {$(2,2,2)$};
        \node (24) at (0,-2) {$(2,4)$};
        \node (8) at (-2,-2) {$(8)$};
        \node (22) at (1,-1) {$(2,2)$};
        \node (4) at (-1,-1) {$(4)$};
        \node (2) at (0,0) {$(2)$};

        \node (33) at (5,-1) {$(3,3)$};
        \node (9) at (4,-1) {$(9)$};
        \node (6) at (3,-1) {$(6)$};
        \node (3) at (4,0) {$(3)$};

        \draw (2) -- (4) -- (8);
        \draw (2) -- (22) -- (222) -- (224);
        \draw (22) -- (24) -- (224);
        \draw (4) -- (24);

        \draw (2) -- (6);
        \draw (3) -- (6);
        \draw (3) -- (9);
        \draw (3) -- (33);

        \draw[dashed] (8) -- (-3,-3);
        \draw[dashed] (24) -- (-1,-3);
        \draw[dashed] (222) -- ++(1,-1);
        \draw[dashed] (6) -- ++(0,-1);
        \draw[dashed] (9) -- ++(0,-1);
        \draw[dashed] (33) -- ++(0,-1);
    \end{tikzpicture}
\end{center}
\caption{The cover graph of divisor-chain partial order}\label{fig:div-chain}
\end{figure}

\subsection{Chromatic polynomials of graphs}
Returning to the category $\cat{Grph}$, we consider a more discerning invariant.
Let $\chi_{\mathcal{G}}(t)$ be the
chromatic polynomial~\cite{Read:chromatic}
of a finite, simple graph $\mathcal{G}$. 
Then $\I :\cat{Grph}_0 \to
\mathbb{Z}[t]$, where $\mathcal{G}\mapsto \chi_{\mathcal{G}}(t)$, is an
isomorphism invariant of $\mono{Grph}$. 

If $\mathcal{G}$ is a graph on $n$ vertices with $m$ edges, then the degree of
$\chi_{\mathcal{G}}(t)$ is $n$ with leading coefficient $1$. Let
$[t^k]\chi_{\mathcal{G}}(t)$ be the coefficient of $t^k$. The coefficients
alternate in sign, and $[t^{n-1}]\chi_{\mathcal{G}}(t) = -m$. Moreover,
$\chi_{\mathcal{G}}(t)$ is the product of chromatic polynomials of each of the
connected components of $\mathcal{G}$. 

We define a partial order $\leq$ on $\Z[t]$ as follows:
\begin{align*} 
    f(t) \leq g(t) \implies \begin{cases} 
        \deg f < \deg g, \text{ or } \\ 
        \deg f = \deg g ~ \text{ and } ~ (\forall k)\left(\left|[t^k]f(t) \right| \leq \left|[t^k]g(t) \right|\right). 
    \end{cases}
\end{align*} 

For an edge $e$ of $\mathcal{G}$, write $\mathcal{G}\setminus e$ for the graph
obtained from $\mathcal{G}$ by deleting the edge $e$, and $\mathcal{G}/e$ for
that obtained from $\mathcal{G}$ by contracting the edge $e$. Recall the
deletion-contraction formula: $\chi_{\mathcal{G}}(t) =
\chi_{\mathcal{G}\setminus e} (t) - \chi_{\mathcal{G}/e} (t)$. We use this
observation to deduce the following.
\begin{prop} \label{prop:chromatic}
    If $\phi : \mathcal{G}_1\hookrightarrow \mathcal{G}_2$, then
    $\chi_{\mathcal{G}_1}(t)\leq \chi_{\mathcal{G}_2}(t)$. 
\end{prop}

As illustrated in Figure~\ref{fig:chromatic-poly}, 
the converse to Proposition~\ref{prop:chromatic} is not true. 
Thus, the partial order $\leq$
is strictly finer than 
that induced from $\I$ via Theorem~\ref{thm:functor}.
Moreover, we do not know precisely the polynomials that arise as chromatic
polynomials, but we can deduce that $\I$ is not surjective.

\begin{figure}[ht]
\begin{subfigure}[b]{0.4\textwidth}
        \centering 
        \begin{tikzpicture}
            \pgfmathsetmacro{\x}{0.75} 
            \pgfmathsetmacro{\y}{0.75}
            \node() at (-0.85*\x, 0.5*\y) {$\mathcal{G}_1$:};
            \node(1) at (0, 0) [circle, fill=black, inner sep=1.5pt] {}; 
            \node(2) at (\x, 0) [circle, fill=black, inner sep=1.5pt] {}; 
            \node(3) at (0, \y) [circle, fill=black, inner sep=1.5pt] {}; 
            \node(4) at (\x, \y) [circle, fill=black, inner sep=1.5pt] {}; 
            \node at (0.5*\x, -\y) {$\chi_1(t) = t^4-3t^3+2t^2$};
            \draw[-] (1) -- (2);
            \draw[-] (2) -- (3);
            \draw[-] (3) -- (1);
        \end{tikzpicture}
    \end{subfigure}~\begin{subfigure}[b]{0.4\textwidth}
\begin{tikzpicture}
            \pgfmathsetmacro{\x}{0.75} 
            \pgfmathsetmacro{\y}{0.75}
            \node() at (-0.85*\x, 0.5*\y) {$\mathcal{G}_2$:};
            \node(1) at (0, 0) [circle, fill=black, inner sep=1.5pt] {}; 
            \node(2) at (\x, 0) [circle, fill=black, inner sep=1.5pt] {}; 
            \node(3) at (\x, \y) [circle, fill=black, inner sep=1.5pt] {}; 
            \node(4) at (0, \y) [circle, fill=black, inner sep=1.5pt] {}; 
            \node at (0.5*\x, -\y) {$\chi_2(t) = t^4-4t^3+6t^2-3t$};
            \draw[-] (1) -- (2);
            \draw[-] (2) -- (3);
            \draw[-] (3) -- (4);
            \draw[-] (4) -- (1);
        \end{tikzpicture}
    \end{subfigure}
    \caption{$\chi_1(t) \leq \chi_2(t)$ but $\mathcal{G}_1$ does not embed into
    $\mathcal{G}_2$}
    \label{fig:chromatic-poly}
\end{figure}

For a graph $\mathcal{G}$, let $\mathcal{G}'$ be obtained from $\mathcal{G}$ by
adjoining an isolated vertex. Thus, $\chi_{\mathcal{G}'}(t) =
t\cdot\chi_{\mathcal{G}}(t)$, and if $\mathcal{H}$ is a graph such that
$\mathcal{G} \hookrightarrow \mathcal{H} \hookrightarrow \mathcal{G}'$, then
$\mathcal{H}$ is isomorphic to either $\mathcal{G}$ or $\mathcal{G}'$. Hence,
$\chi_{\mathcal{G}'}(t)$ covers $\chi_{\mathcal{G}}(t)$. For graphs
$\mathcal{G}_1$ and $\mathcal{G}_2$ with $n_1$ and $n_2$ vertices respectively,
the pseudo-metric $d$ on $\mono{Grph}$ corresponding to $\leq$ is given by 
\begin{align*} 
    d(\mathcal{G}_1, \mathcal{G}_2) &= \left|n_1 - n_2 \right| + \sum_{k\geq 0} \left|[t^k]\left((-t)^{n_1}\chi_{\mathcal{G}_1}(-t^{-1}) - (-t)^{n_2}\chi_{\mathcal{G}_2}(-t^{-1}) \right)\right|.
\end{align*}

\subsection{Conjugacy class sizes}
Finally, we examine an invariant that played a significant role in recent
classifications of finite $p$-groups.

To distinguish groups up to isomorphism, it is often useful to compare the sizes
of conjugacy classes of elements. For nilpotent groups, we usually convert the
group-theoretic property into one for Lie rings. Thus, the conjugacy classes
sizes of a group transform into the ranks of the linear maps given by the
adjoint representation; see ~\cite[Chapter~6]{Khukhro} and
\citelist{\cite{OBrienVoll} \cite{Rossmann:ask2018}}. These ranks are the
\emph{breadths} of elements of a Lie ring, and they played a decisive role in
classifications of certain families of groups of order dividing $p^8$
~\citelist{\cite{p6}\cite{p7}\cite{p8}}. 

Fix a field $K$. A nonassociative $K$-algebra $A$ is a vector space with a
$K$-bilinear product $[x,y]$.  
Let $A$ be class 2 nilpotent, so
every product of three elements in $A$ is $0$. 
Let $[A,A]$ be the subalgebra of $A$
generated by all $[x,y]$ for $x,y\in A$.
We choose a basis $\{e_1,\ldots,e_n\}$ for $A/[A,A]$, so as to identify
(through univalence) $A/[A,A]$ with $K^n$. Similarly, we identify $[A,A]$ with
$K^m$ after choosing a basis $\{f_1,\ldots,f_m\}$. For each $i,j\in \{1,\dots,
n\}$, there exist $\lambda_{ij}^{(k)}\in K$ such that $[e_i, e_j] =
\lambda_{ij}^{(1)}f_1 + \cdots + \lambda_{ij}^{(m)}f_m$. Let $\bm{x}=(x_1,\dots,
x_n)$ be indeterminates, and define an $m\times n$ matrix $M=M_A(\bm{x}) =
(m_{ij}(\bm{x}))$ of linear forms in $K[x_1,\dots,x_n]$ such that 
\begin{align}\label{eqn:matrix-forms}
    m_{ij}(\bm{x}) &= \lambda_{1j}^{(i)}x_1 + \cdots + \lambda_{nj}^{(i)}x_n.
\end{align} 
For $r\geq 1$, let $\mathcal{P}_r(M)\subset K[x_1,\dots, x_n]$ be the set of
$r\times r$ minors of $M$, and let $\mathcal{I}_r(M)$ be the ideal in
$K[x_1,\dots,x_n]$ generated by $\mathcal{P}_r(M)$. We define the closed
projective subscheme $\mathcal{B}_r(M)$ of $\mathbb{P}_K^{n-1}$ given by
$f(x_1,\dots, x_n)=0$ for all $f\in \mathcal{P}_r(M)$. 

Note that $\mathcal{B}_r(M)$ is not an isomorphism invariant of $A$ since we
must choose bases to describe $M$. If $\Phi\in \GL_n(K)$ and $\Gamma\in
\GL_m(K)$, then $(\Phi,\Gamma)$ acts on $M(\bm{x})$ via
$\Gamma^{\mathrm{tr}}M(\Phi \bm{x}) \Phi$. Thus, the ideal $\mathcal{I}_r(M)$ up
to the action of $\GL_n(K)$, via linear substitutions of the variables, is an
isomorphism invariant of $A$. We can consider geometric features of
$\mathcal{B}_r(M)$ such as dimension or the number of $\mathbb{F}_p$-rational
points. 

Fix a prime $p>2$. Let $\cat{pGroup2}$ be the category 
of $p$-groups $G$ of class 2 and exponent $p$. The Lazard
correspondence yields a functor $\Ftr{L}$ from $\cat{pGroup2}$ to the category
of graded Lie $\mathbb{F}_p$-algebras. Let $\Delta$ be the ideals of
$K[x_1,\dots, x_n]$ up to the action of $\GL_n(K)$. For some $r\geq 1$, we
define an isomorphism invariant $\I$ on $\cat{pGroup2}$ by mapping $G$ to
$\mathcal{I}_r(M_{\Ftr{L}(G)})$. Since ideals are considered up to the
$\GL_n(K)$-action, we write instead $\mathcal{I}_r(G)$.

For example, we mimic the abelian group isomorphism invariant partial order from
Section~\ref{sec:abelian-inv} by using the primary decomposition of polynomial
ideals. Let $\mathfrak{p}$ be a prime ideal of $K[x_1,\dots, x_n]$. 
For a
$\mathfrak{p}$-primary ideal $Q$, define $\nu_{\mathfrak{p}}(Q)=\sup\{\ell \mid
Q\subset \mathfrak{p}^{\ell}\}$. If an ideal $J$ of $K[x_1,\dots, x_n]$
has unique minimal primary decomposition
$J = Q_1\cap \cdots \cap Q_r$, then define 
$\hat{\nu}_{\mathfrak{p}}(J) = \nu_{\mathfrak{p}}(Q_i)$ if
$\mathfrak{p}=\sqrt{Q_i}$ and $0$ otherwise. We define a pseudo-metric $d$ on
$\cat{pGroup2}$ via 
\begin{align}\label{eqn:dist-primes}
    d(G,H) & = \sum_{\mathfrak{p}} \left| \hat{\nu}_{\mathfrak{p}}(\mathcal{I}_r(G))- \hat{\nu}_{\mathfrak{p}}(\mathcal{I}_r(H))\right|,
\end{align}
where the sum runs through all prime ideals $\mathfrak{p}$ of $K[x_1,\dots,
x_n]$. We note that in~\eqref{eqn:dist-primes} the
$\hat{\nu}_{\mathfrak{p}}(\mathcal{I}_r(G))$ is determined for a specific
representative of $\mathcal{I}_r(G)$, but its choice does not affect the
distance.

Consider the following family of groups. Fix a polynomial $f(t)\in
\mathbb{F}_p[t]$ of degree $d$ and define a group of order $p^{3d}$:  
\begin{align*}
    H(f) & = \left\{\begin{bmatrix} 1 & a & c\\ & 1 & b \\ & & 1\end{bmatrix}
            ~\middle|~a,b,c\in \mathbb{F}_p[t]/(f(t))\right\} .
\end{align*}
In \cite{Wilson:SkolemNoether} it is shown that $H(f)\cong H(g)$ if, and only
if, $\mathbb{F}_p[t]/(f)\cong \mathbb{F}_p[t]/(g)$.  We explore the distances
between non-isomorphic groups in this family.  The following can be
confirmed using {\sc Magma} \cite{magma}, so we provide just one sample
calculation.

\begin{ex}
    Assume $p>2$ and $f(t)\in \mathbb{F}_p[t]$ is quadratic. Thus,
    $\mathbb{F}_p[t]/(f(t))$ is isomorphic to one of three cases, based on the
    splitting behavior of $f(t)$. In all three cases, $G = H(f)$ has order $p^6$
    and $\Ftr{L}(G) \cong \mathbb{F}_p^4\oplus \mathbb{F}_p^2$. Thus, the
    matrices $M$ defined via~\eqref{eqn:matrix-forms} are $2\times 4$ matrices
    of linear forms in $\mathbb{F}_p[x_1,x_2,x_3,x_4]$. We construct a
    representative of $\mathcal{I}_2(G)$ for each of the three cases. 
    \begin{description}
        \item[Distinct roots in $\mathbb{F}_p$] Set $f_1 = t^2-1$. For some
        choice of basis, 
        \begin{align*}
            M & = 
            \begin{bmatrix} 
                -x_3 & 0 & x_1 & 0 \\ 
                0 & -x_4 & 0 & x_2
            \end{bmatrix}. 
        \end{align*}
        Let $\mathfrak{p}_1 = (x_1,\ x_3)$ and $\mathfrak{p}_2 = (x_2,\ x_4)$ be
        ideals of $\mathbb{F}_p[x_1,x_2,x_3,x_4]$. Then 
        \begin{align*}
            \mathcal{I}_2(H(f_1)) & = (x_1 x_2,\ x_1 x_4,\ x_2 x_3,\ x_3 x_4) = \mathfrak{p}_1 \cap \mathfrak{p}_2.
        \end{align*}

        \item[Repeated root in $\mathbb{F}_p$] Set $f_2 = t^2$. Then 
        \[
            \mathcal{I}_2(H(f_2))=(x_3^2,\ x_1 x_4-x_2x_3,\ x_1x_3,\ x_1^2)
            \subset \mathfrak{p}_1,
        \]
        so $\mathcal{I}_2(H(f_2))$ is $\mathfrak{p}_1$-primary.

\vspace*{0.1cm}
        \item[No roots in $\mathbb{F}_p$] Set $f_3=t^2-\omega$, for some
        nonsquare $\omega\in\mathbb{F}_p$. Then
        \[ 
            \mathcal{I}_2(H(f_3))=(x_3^2-\omega x_4^2,\ x_1x_4-x_2 x_3,\ x_1x_3-\omega x_2 x_4,\ x_1^2-\omega x_2^2)
        \] 
        is a prime ideal. Write $\mathfrak{p}_3 = \mathcal{I}_2(H(f_3))$.
    \end{description}

    Under the action of $\GL_4(\mathbb{F}_p)$ on the prime ideals of
    $\mathbb{F}_p[x_1,x_2,x_3,x_4]$, both $\mathfrak{p}_1$ and $\mathfrak{p}_2$
    are in the same orbit, and $\mathfrak{p}_3$ is in a different orbit. To see
    this, note that, after a harmless relabelling of variables, the action of
    $\mathrm{GL}_4(\mathbb{F}_p)$ maps $\mathfrak{p}_1$ to $(x_1 + ax_3 + bx_4,\
    x_2 + cx_3 + dx_4)$ for some $a,b,c,d\in \mathbb{F}_p$, and such an ideal 
    never equals $\mathfrak{p}_3$ even after relabelling the variables.

    Let $G_i = H(f_i)$ for $i\in \{1,2,3\}$. Following ~\eqref{eqn:dist-primes},
    we deduce that 
    \[ 
        d(G_1,G_2)=|\hat{\nu}_{\mathfrak{p}_1}(\mathcal{I}_2(G_1))-\hat{\nu}_{\mathfrak{p}_1}(\mathcal{I}_2(G_2))|+|\hat{\nu}_{\mathfrak{p}_2}(\mathcal{I}_2(G_1))-\hat{\nu}_{\mathfrak{p}_2}(\mathcal{I}_2(G_2))|=1.
    \]
    The other distances are computed similarly:
        $d(G_2, G_3) = 2$ and 
        $d(G_1, G_3) = 3$.
\end{ex}

To summarize, the addition of partial ordering and distance to isomorphism
invariants makes it possible not only to distinguish groups, but to organize
them geometrically according to distance.  For example, the groups $G_1$ and
$G_2$ are closer to each other than either is to $G_3$, but the latter is closer
to $G_2$ than to $G_1$. 
 
\section*{Acknowledgements}

We thank Michael Shulman for remarks on Homotopy Type Theory and Anton
Baykalov and Tobias Rossmann
for comments on a draft of the paper. Maglione was supported by DFG grant VO
1248/4-1 (project number 373111162) and DFG-GRK 2297. O'Brien was supported by
the Marsden Fund of New Zealand grant UOA 107. Wilson was supported by a Simons
Foundation Grant, identifier \#636189.

\bibliographystyle{abbrv}

\end{document}